\documentclass{amsart}
\usepackage[utf8]{inputenc}
\usepackage{amssymb,latexsym}
\usepackage{amsmath}
\usepackage{graphicx}
\usepackage{textcomp}
\usepackage[dvipsnames]{xcolor}

\usepackage{amsthm,amssymb,enumerate,graphicx, tikz}
\usepackage{amscd}
\usepackage{setspace}
\usepackage{comment}
\usepackage{hyperref}
\usepackage{cleveref}

\def\@logofont{\footnotesize}
\def\@setaddresses{\par
  \nobreak \begingroup
  \footnotesize
  \def\author##1{\nobreak\addvspace\bigskipamount}%
  \def\\{\par\nobreak}%
  \interlinepenalty\@M
  \def\address##1##2{\begingroup
    \par\addvspace\bigskipamount\indent
    \@ifnotempty{##1}{(\ignorespaces##1\unskip) }%
    {\scshape\ignorespaces##2}\par\endgroup}%
  \def\curraddr##1##2{\begingroup
    \@ifnotempty{##2}{\nobreak\indent\curraddrname
      \@ifnotempty{##1}{, \ignorespaces##1\unskip}\/:\space
      ##2\par}\endgroup}%
  \def\email##1##2{\begingroup
    \@ifnotempty{##2}{\nobreak\indent\emailaddrname
      \@ifnotempty{##1}{, \ignorespaces##1\unskip}\/:\space
      \ttfamily##2\par}\endgroup}%
  \def\urladdr##1##2{\begingroup
    \def~{\char`\~}%
    \@ifnotempty{##2}{\nobreak\indent\urladdrname
      \@ifnotempty{##1}{, \ignorespaces##1\unskip}\/:\space
      \ttfamily##2\par}\endgroup}%
  \addresses
  \endgroup
}
\renewcommand*\subjclass[2][2010]{%
  \def\@subjclass{#2}%
  \@ifundefined{subjclassname@#1}{%
    \ClassWarning{\@classname}{Unknown edition (#1) of Mathematics
      Subject Classification; using '2000'.}%
  }{%
    \@xp\let\@xp\subjclassname\csname subjclassname@#1\endcsname
  }%
}

\usepackage{graphicx,amsmath,amssymb}
\usepackage{algorithm}
\usepackage{mathdots, comment}
\usepackage[noend]{algpseudocode}

\newtheorem{theorem}{Theorem}[section]

\newtheorem*{theorem*}{Theorem}

\newtheorem{lemma}[theorem]{Lemma}

\newtheorem{question}[theorem]{Question}

\newtheorem{conjecture}[theorem]{Conjecture}
\theoremstyle{definition}

\theoremstyle{remark}
\newtheorem{remark}[theorem]{Remark}
\newtheorem{example}[theorem]{Example}

\begin{document}
\title{The weak acyclic matching property in abelian groups}


\author[M. Aliabadi and P. Taylor]{Mohsen Aliabadi$^{1}$ \and Peter Taylor$^2$}
\thanks{$^1$Department of Mathematics, Clayton State University, 2000 Clayton State Boulevard, Morrow, GA 30260, USA.  \url{maliabadisr@clayton.edu}, ORCID: 0000-0001-5331-2540\\
$^2$Independent researcher, Valencia, Spain. \url{pjt33@cantab.net}, ORCID: 0000-0002-0556-5524 }

\begin{abstract}
 A matching from a finite subset $A\subset\mathbb{Z}^n$ to another subset $B\subset\mathbb{Z}^n$ is a bijection $f : A \rightarrow B$ with the property that $a+f(a)$ never lies in $A$. A matching is called acyclic if it is uniquely determined by its multiplicity function. Alon et al. established the acyclic matching property for $\mathbb{Z}^n$, which was later extended to all abelian torsion-free groups. In a prior work, the authors of this paper settled the acyclic matching property for all abelian groups. The objective of this note is to explore a related concept, known as the weak acyclic matching property, within the context of abelian groups.
\end{abstract}

\maketitle

\section{Introduction}

Let $ (G,+) $ denote an abelian group. Consider nonempty finite subsets $ A, B \subseteq G $. A \textit{matching} from $ A $ to $ B $ is defined as a bijection $ f : A \rightarrow B $ satisfying the condition $ a + f(a) \notin A $ for every $ a \in A $. When such a matching exists from $ A $ to $ B $, we say that $ A $ is \textit{matched} to $ B $. This concept, introduced by Fan and Losonczy in \cite{Fan}, serves as a tool in exploring Wakeford's classical problem concerning canonical forms for symmetric tensors \cite{Wakeford}. It is evident that the existence of a matching from $ A $ to $ B $ requires $ |A| = |B| $ and $ 0 \notin B $.

An abelian group $ G $ possesses the \textit{matching property} if the aforementioned conditions on $ A $ and $ B $ suffice to ensure the existence of a matching from $ A $ to $ B $. Losonczy addressed the question of which groups possess this property in \cite{Losonczy}, stating:
\begin{theorem}\label{matching property}
    An abelian group $ G $ possesses the matching property if and only if $ G $ is torsion-free or cyclic of prime order.
\end{theorem}
For any matching $ f:A\rightarrow B $, the associated multiplicity function $ m_f : G \rightarrow \mathbb{Z}_{\geq 0} $ is defined as follows:
\begin{align*}
  \forall x \in G, m_f(x) = |\{a \in A : a + f(a) = x\}|.   
\end{align*}

A matching $ f : A \rightarrow B $ is termed \textit{acyclic} if for any matching $ g: A \rightarrow B $, $ m_f = m_g $ implies $ f = g $. If an acyclic matching from $ A $ to $ B $ exists, we say that $ A $ is \textit{acyclically matched} to $ B $. A group $ G $ exhibits the \textit{acyclic matching property} if for any pair of subsets $ A $ and $ B $ in $ G $ with $ |A| = |B| $ and $ 0 \notin B $, there is at least one acyclic matching from $ A $ to $ B $. Alon et al. proved in \cite{Alon} that the additive group $ \mathbb{Z}^n $ possesses the acyclic matching property. Losonczy \cite{Losonczy} extended this result to all abelian torsion-free groups by exploiting a total ordering compatible with the structure of abelian torsion-free groups, established by Levi \cite{Levi}.

In \cite{Aliabadi 0}, it was noted that primes $p$ for which $\mathbb{Z}/p\mathbb{Z}$ lacks the acyclic matching property have a density of at least $\frac{7}{24}$. Moreover, building upon the rectification principle introduced by \cite{Lev}, which suggests that a sufficiently small subset of $\mathbb{Z}/p\mathbb{Z}$ can be embedded in integers while preserving certain additive properties, it was demonstrated in \cite{Aliabadi 0} that for sets $A,B\subset \mathbb{Z}/p\mathbb{Z}$, $A$ is acyclically matched to $B$ provided that $0$ is not an element of $B$ and $|A|=|B|\leq\sqrt{\log_2p}-1$.

The present authors \cite{Taylor} provided a necessary and sufficient condition for abelian groups to possess the acyclic matching property:
\begin{theorem}\label{nes-suf}
    An abelian group \( G \) possesses the acyclic matching property if and only if either \( G \) is torsion-free or \( G = \mathbb{Z}/p\mathbb{Z} \), where \( p \in \{2,3,5\} \).
\end{theorem}
 See \cite{Hamidoune} for insights into the enumeration aspects of matchings. Linear analogues of matching theory are developed in \cite{Eliahou}, while a matroidal version is presented in \cite{Zerbib}. In particular, the linear version of acyclic matchings is studied in \cite{Aliabadi 0}. From the viewpoint of additive number theory, connections to matching theory can also be found in \cite{Lev 0}.\\
 
\textbf{Acyclic matchings for specific sets:} When discussing matchings between two sets $A$ and $B$ within $G$, two scenarios represent the extremes of the spectrum in terms of the number of matchings:
\begin{enumerate}
    \item $A+B=A$: In this case $A$ is not matched to $B$. This case is treated in \cite[Lemma 2.3] {Aliabadi 3} where it is proved that in this situation $B$ must be a subgroup of $G$, and $A$ is a coset of $B$.
    \item $A\cap (A+B)=\emptyset$: In this case, not only is $A$ matched to $B$ but every bijection from $A$ to $B$ constitutes a matching, resulting in the maximum number of matchings.
    \end{enumerate}
Motivated by the preceding discussion, the acyclic matching problem arises for sets $A$ and $B$ where $A\cap (A+B)=\emptyset$. Intuitively, the situation might be complex when $A\cap (A+B)=\emptyset$ because there could be as many as $|A|!$ matchings, increasing the probability of at least one acyclic matching existing. On the other hand, the number of matchings may grow faster than the number of multiplicity functions, so the average number of matchings per multiplicity function may be large reducing the probability of at least one acyclic matching existing. The following special cases of acyclic matching were proved in \cite{Aliabadi 0}:
\begin{theorem}
Let $A$ be a subset of the cyclic group $\mathbb{Z}/p\mathbb{Z}$ where $p$ is a prime number. Suppose $A$ satisfies $A\cap 2A = \emptyset$ where $2A=\{2a \mid a\in A\}$and $A$ is of size $k$ with $k\cdot 2^{k-1} < p$. Then $A$ is acyclically matched to itself via the identity map.
\end{theorem}
\begin{theorem}
Let $A$ and $B$ be subsets of an abelian group $G$. Suppose $A$ and $B$ are of the same size satisfying $A \cap (A + B) = \emptyset$. Then there exists an acyclic matching $f : A \rightarrow B$ if $B$ is a Sidon set. (That is, the equation $x + y = z + w$ has no solution in $B$ with $\{x, y\} \cap \{z, w\} = \emptyset$.)
\end{theorem}

Following \cite{Aliabadi 2}, we define an abelian group $G$ to possess the \textit{weak acyclic matching property} if, for every finite nonempty subsets $A$ and $B$ of $G$ with $|A|=|B|$ and $A\cap (A+B)=\emptyset$, $A$ is acyclically matched to $B$. In \cite{Aliabadi 2}, the following conjecture concerning the weak acyclic matching property was proposed.
\begin{conjecture}\label{w.a.m.p}
    There are infinitely many $n$ for which $\mathbb{Z}/n\mathbb{Z}$ has the weak acyclic matching property.
\end{conjecture}
The following question was also raised in \cite{Aliabadi 0}.
\begin{question}\label{w.a.m}
Does $\mathbb{Z}/p\mathbb{Z}$ possess the weak acyclic matching property for every prime $p$?
\end{question}
In this note, our main result, stated as Theorem ~\ref{main}, provides an affirmative answer to Conjecture ~\ref{w.a.m.p} and thoroughly addresses Question ~\ref{w.a.m}. Furthermore, Theorem ~\ref{main} classifies all abelian groups with respect to their adherence to the weak acyclic matching property.

\begin{remark}
  It is worth noting that while the matching theory developed in this paper is related to the classical theory introduced by Philip Hall in 1935, the two are conceptually distinct. In particular, to address a linear algebra problem posed by Wakeford~\cite{Wakeford}, Fan~\cite{Fan} constructed a bipartite graph \( \mathcal{G} = (V(\mathcal{G}), E(\mathcal{G})) \) based on two sets \( A, B \subset \mathbb{Z}^n \), which may overlap. Each element \( a \in A \) is assigned a symbol \( x_a \), and each \( b \in B \) a symbol \( y_b \), forming vertex sets
\[
X = \{x_a : a \in A\} \quad \text{and} \quad Y = \{y_b : b \in B\}.
\]
The graph has vertex set \( V(\mathcal{G}) = X \cup Y \), with an edge between \( x_a \in X \) and \( y_b \in Y \) if and only if \( a + b \notin A \).

In graph-theoretic terms, a \emph{matching} is a set of edges with no shared endpoints, and a \emph{perfect matching} is one that covers all vertices. In the bipartite graph associated with \( A \) and \( B \), perfect matchings correspond to bijections \( f : A \to B \) such that \( a + f(a) \notin A \) for all \( a \in A \). Wakeford’s question thus reduces to determining the existence of such perfect matchings, known as acyclic matchings from \( A \) to \( B \), discussed earlier.

\end{remark}

\section{Main results}
\textbf{Multiplicity functions in a more general setting:} We begin with a few definitions. Let $A$, $B$ and $C$ be three nonempty sets. An operator $\oplus: A \times B \to C$ is said to possess the {\it{left cancellation property}} provided:
$$
a \oplus b_1 = a \oplus b_2 \iff b_1 = b_2,
$$
for every $a\in A$ and $b_1, b_2\in B$.\\
It will be convenient to define also the operator $\ominus:C \times A \to B$ as an {\it{inverse}} operation: 
$$
c \ominus a = b \iff a \oplus b = c,
$$
where $a\in A$, $b\in B$ and $c\in C$.\\
Previously, we introduced a multiplicity function associated with a matching. We may consider a similar definition of a multiplicity function associated with a bijection. Let $A$, $B$ and $C$ be three nonempty sets. Let $\oplus:A\times B\to C$ be an operator. Given a bijection $f:A\to B$, the associated multiplicity function $ m_f : C \rightarrow \mathbb{Z}_{\geq 0} $ is defined as
\begin{align*}
  \forall x \in C, m_f(x) = |\{a \in A : a \oplus f(a) = x\}|. 
\end{align*}

In the following lemma, we obtain a bijection with a unique multiplicity function.

\begin{lemma}\label{main-lemma}
    Let $A, B$ be finite sets of the same cardinality. Let $C$ be a set. Let $\oplus: A \times B \to C$ be an operator with the left cancellation property and also the operator $\ominus$ as an inverse operation. Consider the set $F$ of bijections $f: A \to B$ with corresponding multiplicity functions $m_f(c) = |\{ a \in A : a \oplus f(a) = c \}|$. Then there exists an element $f_0 \in F$ which has a unique $m_{f_0}$. That is, if $g \in F$ and $m_g = m_{f_0}$ then $g = f_0$.
\end{lemma}
\begin{proof}
    The proof begins by constructing such an $f_0$. Let $C' = \{ a \oplus b : (a, b) \in A \times B \}$ be the image of $\oplus$. Since $A$ and $B$ are finite, $C'$ is also finite and we can label its elements $c_1, \ldots, c_k$ where $k = |C'|$.

    Let $A_1 = A, B_1 = B$. For each $1\leq j\leq k$, let $A'_j = \{ a \in A_j|\hspace{0.1cm} \exists b \in B_j : a \oplus b = c_j \}$. For each $a \in A'_j$ we assign $f_0(a) = c_j \ominus a$. For the next iteration we define $A_{j+1} = A_j \setminus A'_j$ and $B_{j+1} = B_j \setminus \{ c_j \ominus a : a \in A'_j\}$. 

    Since the paired elements of $A$ and $B$ are removed from the active subsets $A_{j+1}$, $B_{j+1}$ it is evident that no element is paired twice. Since $|A| = |B|$ and $|A'_j| = |B'_j|$, by induction we always have $|A_j| = |B_j|$.  We claim that $|A_{k+1}|=0$. Suppose for contradiction that $|A_{k+1}| > 0$. Then, we can arbitrarily choose $a' \in A_{k+1}$ and $b' \in B_{k+1}$. Since there exists $j \le k$ for which $a' \oplus b' = c_j$, at least one of them should have been removed no later than at step $j$, which leads to a contradiction. Hence $|A_{k+1}| = 0$. Therefore $f_0$ is a bijection $A \to B$, as claimed.

    We show the uniqueness of $m_{f_0}$ by assuming that $g:A \to B$ is a bijection and proving the property $$\left(\forall 1 \le i < j: m_g(c_i) = m_{f_0}(c_i)\right) \implies \left( \forall 1 \le i < j: \forall a \in A'_i: g(a) = f_0(a) \right)$$ by induction on $j$.

    When $j = 1$ the right-hand side of the implication is trivially true, giving the base case.

    For the inductive step we note that, by construction, $f_0$ maximizes $m_{f_0}(j)$ subject to the assignments made in steps $1$ to $j-1$. Therefore if $g$ agrees with $f_0$ on all the assignments made in those steps and $m_g(c_j) = m_{f_0}(c_j)$, $g$ must make the same assignments as those made for $f_0$ in step $j$. It cannot assign an $a$ which is not in $A_j$ because it is already assigned; the values in $A_j$ that can potentially be assigned to yield $c_j$ are precisely those in $A'_j$; and each has a unique  $b \in B_j$ (and, indeed, $b \in B$) to which it can be assigned in order to yield $c_j$, because of the left cancellation property.
    Therefore, this property holds for all \( j \), and we observe that when specialized to \( j = k+1 \), it implies that if \( m_g = m_{f_0} \), then \( g = f_0 \).

\end{proof}

\begin{example}
    We implement the construction described in Lemma ~\ref{main-lemma} to find the acyclic matching $f_0 : A \to B$, where $A = \{0, 1, 2, 7\}$ and $B = \{3, 4, 9, 10\}$ are considered as subsets of $\mathbb{Z}/13\mathbb{Z}$ with $\oplus$ as the usual group operation. Then $C' = \{3, 4, 5, 6, 9, 10, 11, 12\}$ and we can take the natural ordering.

    $A_1 = \{0, 1, 2, 7\}$, $B_1 = \{3, 4, 9, 10\}$, $c_1 = 3$. We obtain $A'_1 = \{0, 7\}$ and assign $f_0(0) = 3$ and $f_0(7) = 9$.

    $A_2 = \{1, 2\}$, $B_2 = \{4, 10\}$, $c_2 = 4$. We obtain $A'_2 = \emptyset$ and make no assignments.

    $A_3 = \{1, 2\}$, $B_3 = \{4, 10\}$, $c_3 = 5$. We obtain $A'_3 = \{1\}$ and assign $f_0(1) = 4$.

    For $4 \le j \le 8$ we have $A_j = \{2\}$, $B_j = \{10\}$. Finally, with $c_8 = 12$ we obtain $A'_7 = \{2\}$ and assign $f_0(2) = 10$.
\end{example}

In the following theorem, we will establish that all abelian groups possess the weak acyclic matching property.
\begin{theorem}\label{main}
    Every abelian group possesses the weak acyclic matching property.
\end{theorem}
\begin{proof}
  Let $G$ be an abelian group. Let $A$ and $B$ be nonempty finite subsets of $G$ with $|A|=|B|$ and $A\cap (A+B)=\emptyset$. Since the bijections $A \to B$ are precisely the matchings and the group operation is cancellative, we can apply lemma ~\ref{main-lemma} to obtain an acyclic matching from $A$ to $B$, completing the proof.
\end{proof}
\begin{remark}
    In the case that $A \subseteq C$ and $B \subseteq C$, any operation $\oplus: A \times B \to C$ for which $(C, \oplus)$ is a quasigroup has the left cancellation property. The definition of matching generalizes straightforwardly to quasigroups, and so does the proof of Theorem ~\ref{main}. With that in mind, one might explore the properties of matchings in broader contexts beyond abelian groups, such as quasigroups. Delving further into this line of research could prove to be worthwhile.
\end{remark}

\textbf{Data sharing:} Data sharing not applicable to this article as no datasets were generated or analysed.\\
\textbf{Conflict of interest:} To our best knowledge, no conflict of interests, whether of financial or personal nature, has influenced the work presented in this article.

\end{document}